\documentclass[a4paper, 12pt]{article}
\usepackage{amsmath}
\usepackage{amsfonts}
\usepackage{amsthm}
\usepackage{array}
\usepackage{url}
\usepackage{hyperref}
\usepackage{cite} 
\usepackage{authblk}

\theoremstyle{plain}
\newtheorem{thm}{Theorem} 

\newtheorem{lem}{Lemma}

\theoremstyle{remark}
\newtheorem*{remarks*}{Remarks}

\newtheorem*{rem*}{Remark}

\theoremstyle{definition}
\newtheorem{defn}{Definition}

\allowdisplaybreaks

\title{Residual periodicity on the Markoff Surface}

\author{Solomon Vishkautsan\footnote{author is supported by the ERC-Grant ``Diophantine Problems," No. 267273.}
}
\affil{\small{Scuola Normale Superiore, Piazza dei Cavalieri 7, 56126 Pisa (Italy)\\e-mail: \textit{solomon.vishkautsan@sns.it}}}

\date{July 2015}

\makeatletter
\let\thetitle\@title
\let\theauthor\@author
\makeatother

\begin{document}
\maketitle

\begin{abstract}
A case study of arithmetic dynamics over the rationals on the Markoff surface is presented, in particular the local-global dynamical property of strong residual periodicity. 
The dynamical system induced by the composition of any two of the reflections from the three special points at infinity on the Markoff surface is shown to be strongly residually periodic. This residual periodicity is explained by the existence of periodic conic sections of the Markoff surface with no rational points. It is also proven that cutting these conic sections from the surface eliminates strong residual periodicity.
\end{abstract}

\section{Introduction}

\subsection{The Markoff surface}
\quad \\
We study the \emph{Markoff surface}, given by the affine equation
\begin{equation} \label{eq:markoff-with-3}
M: \quad x^2+y^2+z^2=3xyz
\end{equation}
from an arithmetic dynamics perspective. Classically this surface was first studied by Markoff~\cite{article:markoff1880}, who proved that there are infinitely many integral points, generated from the point $(1,1,1)$ by repeated applications of three special automorphisms of the surface (see \eqref{eq:reflections} below). This surface has many connections to problems in Diophantine approximation and Diophantine geometry (e.g., Cassels~\cite{book:cassels1957}, Zagier~\cite{article:zagier1982}, Silverman~\cite{article:silverman1990}, Baragar~\cite{article:baragar1996} and Corvaja and Zannier~\cite{article:corvaja-zannier2006}, just to state a few).

In terms of dynamics over $\mathbb{Q}$ (as opposed to the dynamical behavior of the integral points on $M$), the constant $3$ in \eqref{eq:markoff-with-3} can be replaced with any $a\in\mathbb{Q}\setminus\{0\}$, as these surfaces are all isomorphic over $\mathbb{Q}$. Moreover, these isomorphisms will only affect the dynamical behavior modulo $p$ at a finite number of primes $p$, so that strong residual periodicity (see \S\ref{subsec:residual-periodicity} below) is invariant under them. In fact, for what follows, we replace $3$ with $1$, to make calculations simpler; i.e., we take $M$ to be
\begin{equation} \label{eq:markoff-with-1}
M: \quad x^2+y^2+z^2=xyz.
\end{equation}

Geometrically, the Markoff surface is a singular cubic surface, with a unique singular point at $(0,0,0)$. It has three lines at infinity, intersecting at three distinct points. The following automorphisms of $M$ are of particular interest:
\begin{equation} \label{eq:reflections}
\begin{aligned}
\phi_1: (x,y,z) \mapsto (yz-x,y,z), \\
\phi_2: (x,y,z) \mapsto (x,xz-y,z), \\
\phi_3: (x,y,z) \mapsto (x,y,xy-z). 
\end{aligned}
\end{equation}
Each automorphism $\phi_i$, $i=1,2,3$, is an involution of the surface $M$; i.e.\ $\phi_i^2=\operatorname{id}$. Geometrically, each $\phi_i$ is a reflection through one of the three intersection points of the three lines at infinity of $M$: 
each line (not at infinity) going through one of the three points, intersects the surface $M$ at two other points (counted with multiplicity); the respective automorphism $\phi_i$ interchanges between these two points. Together with the projective automorphisms of the surface, these reflections generate the group $\mathcal{A}$ of birational automorphisms of the (projective closure of the) surface that restrict to automorphisms on the affine part of the surface. 
The subgroup of $\mathcal{A}$ generated by $\phi_1,\phi_2,\phi_3$ is isomorphic to the free product
\begin{equation}
\left<\phi_1\right>*\left<\phi_2\right>*\left<\phi_3\right>.
\end{equation}
In particular, the product of any two of the three reflections is an automorphism of infinite order of the Markoff surface (cf. \`El'-Huti~\cite{article:el-huti1974}). 

Cantat and Loray~\cite{article:cantat-loray2009} and Cantat~\cite{article:cantat2009} studied holomorphic and real dynamics of the automorphisms of a family of surfaces of the form
\begin{equation}
S_{(A,B,C,D)}: x^2+y^2+z^2+xyz=Ax+By+Cz+D
\end{equation}
(the Markoff surface is the case $A=B=C=D=0$ with a change of variables). In particular, they proved that the only real periodic point of automorphisms of infinite order of the Markoff surface, is the fixed and singular point $(0,0,0)$ (cf.\ proposition 3.2 and \S{5.1} in~\cite{article:cantat-loray2009} and table 1 in~\cite{article:cantat2009}).

In this article we study the arithmetic dynamics of automorphisms of the form $\phi_i\phi_j$ (for distinct $i$ and $j$) of the Markoff surface, and obtain results on a local-global arithmetic dynamics property called \textit{strong residual periodicity}, which is explained in the next subsection.

\subsection{Residual periodicty and main results} \label{subsec:residual-periodicity}
\quad \\
We define an \textit{arithmetic-geometric dynamical system} to be a pair $\mathcal{D}=(X, f)$, where $X\subseteq\mathbb{P}^N$ is a quasiprojective variety defined over a number field $K/\mathbb{Q}$, with a given embedding in projective space, and $f:X\rightarrow{X}$ is an endomorphism of $X$. The dynamics of this system are induced by iterating $f$ on rational points $P\in{X(K)}$, i.e.\ the orbits
\begin{equation}
P, f(P), f(f(P)), f(f(f(P))), \ldots .
\end{equation}
A point $P\in{X(K)}$ is called \textit{periodic} if $f^n(P) = P$ for some positive integer $n$. The minimal such $n$ is called the \textit{exact period} of $P$.

For all but finitely many primes in the number field $K$ (these primes correspond to the primes of good reduction of both $X$ and $f$, cf.\ Silverman~\cite{book:silverman2007}, Hutz~\cite[\S{2}]{article:hutz2009} and Bandman, Grunewald and Kunyavski{\u\i}~\cite[\S{6}]{article:bandman-grunewald-kunyavskii2010}), we can reduce the dynamical system $\mathcal{D}$ modulo $p$, to obtain a \textit{residual dynamical system} $\mathcal{D}_p=(X_p, f_p)$ defined over the residue field $\mathbb{F}_p$. We can then compare the dynamics of the ``global" system $\mathcal{D}$, and the dynamics of all (but finitely many of) the ``local" systems $\mathcal{D}_p$.

A periodic point $P\in X(K)$ of period $n$ guarantees the existence of periodic points of period at most $n$ in all but finitely many residual systems (In fact much more can be said, cf.\ Silverman~\cite[Theorem 2.21]{book:silverman2007} and Hutz~\cite{article:hutz2009}). The converse (in a similar vein to counterexamples to the Hasse principle, cf.\ Cohen~\cite[\S{5.7}]{book:cohen2007}) is false however, as can be seen from the relatively simple example of the dynamical system induced by the polynomial 
\begin{equation} \label{eq:bgk-srp}
f(x)=(x^2-2)(x^2-3)(x^2-6) + x
\end{equation}
on the affine line $\mathbb{A}^1_\mathbb{Q}$; this system has no $\mathbb{Q}$-periodic points, but has a fixed point modulo every prime $p\in\mathbb{Z}$ (cf.\ Bandman, Grunewald and Kunyavski{\u\i}~\cite{article:bandman-grunewald-kunyavskii2010} for details). 

We give a definition generalizing the property shown in the last example: 
\begin{defn} \label{def:srp}
Let $\mathcal{D}=(X, f)$ be an arithmetic-geometric dynamical system defined over a number field $K$ and let $\mathcal{F}$ be an algebraic subset of $X$ also defined over $K$ (we call $\mathcal{F}$ the \textit{forbidden set}). We say that $\mathcal{D}$ is \textit{strongly residually periodic} (with respect to $K$ and $\mathcal{F}$) if there exist a positive integer $M$ and a finite set of primes $\mathcal{S}$ in $K$, such that for any prime $p\notin\mathcal{S}$ the residual system $\mathcal{D}_p$ has a $\mathbb{F}_p$-periodic point $P$ of period at most $M$ and such that $P$ is not in the reduction modulo $p$ of the forbidden set $\mathcal{F}$ (i.e.\ for all but finitely many primes, there exist points of universally bounded period in the residual systems lying away from the reductions of the forbidden set).
\end{defn}

If $\mathcal{D}$ is strongly residually periodic with a bound $M$, we denote for short that $\mathcal{D}$ is $SRP(M)$. We also modify the definition of a dynamical system by adding a third component: $\mathcal{D}=(X, f, \mathcal{F})$ where $X$ is the variety, $f$ is the endomorphism and $\mathcal{F}$ is the chosen forbidden set. The companion residual systems are likewise modified: $\mathcal{D}_p=(X_p, f_p, \mathcal{F}_p)$. The introduction of the forbidden set allows us for example to consider the dynamical system $\mathcal{D}=(\mathbb{A}^1_\mathbb{Q}, f(x)=x^2, \mathcal{F}=\{0,1\})$ as being \emph{not} strongly residually periodic (over $\mathbb{Q}$), even though $0$ and $1$ are fixed modulo every prime $p$ in $\mathbb{Q}$ (proving that the system $\mathcal{D}$ above is not strongly residually periodic is similar to Example 6.6 in Bandman, Grunewald and Kunyavski{\u\i}~\cite{article:bandman-grunewald-kunyavskii2010}). 

\textit{Strong residual periodicity} was first introduced by Bandman, \\ Grunewald and Kunyavski{\u\i}~\cite{article:bandman-grunewald-kunyavskii2010}, and results on residual periodicity on smooth cubic surfaces were published by the author of the present article~\cite{article:wishcow2014}. A related study by Silverman~\cite{article:silverman2008}, generalized by Akbary and Ghioca~\cite{article:akbary-ghioca2009}, proved that the orbit lengths of reductions of a non-periodic point are not universally bounded (in the sense of Definition \ref{def:srp}), implying that strong residual periodicity cannot be explained by a finite set of rational points in $X(K)$ (In fact these papers provide an asymptotic lower bound for the growth of the residual orbit lengths with respect to the primes).

We now present the main results of this article for the automorphisms $\phi_i\phi_j$ (see \eqref{eq:reflections}) of the Markoff surface $M$ (see \eqref{eq:markoff-with-1}):

\begin{thm} \label{thm:main-thmA}
The dynamical system $\mathcal{D}_1 = (M, \phi_i\phi_j, \mathcal{F}_1=\{(0,0,0)\})$, where $i,j\in\{1,2,3\}$ are distinct, is SRP(3) over $\mathbb{Q}$ and the universal bound of $3$ is sharp.
\end{thm}

\begin{thm} \label{thm:main-thmB}
The dynamical system $\mathcal{D}_2 = (M, \phi_1\phi_2, \mathcal{F}_2)$, where the forbidden set
$$\mathcal{F}_2 = \left\{M\cap \{(z^2-1)=0\}\right\} \cup \{(0,0,0)\}$$ 
is the union of $2$ conic sections on the surface, is \emph{not} strongly residually periodic over $\mathbb{Q}$.
\end{thm}

%




\begin{rem*}
Theorem~\ref{thm:main-thmB} can be generalized to any $\phi_i\phi_j$, but we state it only for $\phi_1\phi_2$ for clarity.
\end{rem*}

Theorem~\ref{thm:main-thmB} is proved by showing that for a specific increasing sequence of primes, the minimal periods of the residual systems $\mathcal{D}_{2,p}$ (i.e.\ the minimum of exact periods of $\mathbb{F}_p$-periodic points in $\mathcal{D}_{2,p}$) are unbounded over the primes (see Lemmas \ref{lem:two} and \ref{lem:three} in \S{\ref{sec:local}} below). It should be noted that the conics defined by the hyperplane sections $z\pm{1}$ are fixed under the third iteration of $\phi_1\phi_2$, but have no $\mathbb{Q}$-rational points. For a dynamical system $\mathcal{D}=(X, f)$ we will call a subvariety of ${X}$ \emph{strongly periodic} if it is fixed under some iteration of the endomorphism $f$ (this terminology is not standard).

Now, the results of this article, together with the examples in \cite[\S{9}]{article:wishcow2014}, may make one suspect that strong residual periodicity of dynamical systems on \emph{surfaces} serves merely as an indicator for the existence of absolutely irreducible strongly periodic curves, defined over the base field $K$, having only finitely many $K$-rational points. As a counterexample, one may take a trivial generalization of example \eqref{eq:bgk-srp} to the affine plane by fixing a second parameter. A less trivial generalization of example \eqref{eq:bgk-srp} was suggested to the author by St\'ephane Lamy:
\begin{equation}\label{eq:lamy}
f(x,y) = (y,(y^{2}-2)(y^{2}-3)(y^{2}-6)+x).
\end{equation} 
The map $f$ is a H\'enon automorphism of the affine plane, and as such has no periodic curves, i.e.\ curves which are invariant under some iteration of $f$ (cf.\ Bedford and Smillie~\cite{article:bedford-smillie1991}). However, one can check that the automorphism $f$ has no $\mathbb{Q}$-periodic points and has a fixed point modulo every prime integer $p$, i.e.\ it is strongly residually periodic: The first fact can be shown for example by using bounds on valuations as in Ingram~\cite{article:ingram2011}. The second fact reduces to example \eqref{eq:bgk-srp}  by looking at the diagonal $x=y$. 


\section{Preliminaries}

Let $\mathcal{D}=(X, f, \mathcal{F})$ be a dynamical system as in $\S\ref{subsec:residual-periodicity}$. Let $\sigma\in PGL_{n+1}(K)$ be a projective automorphism that restricts to an automorphism of $X$,  then the \emph{linear conjugation} of $f$ by $\sigma$ is $f^\sigma = \sigma^{-1}\circ{f}\circ\sigma$. The dynamical system $\mathcal{D}^\sigma = (X, f^\sigma, \sigma^{-1}\left(\mathcal{F}\right))$ has the same dynamics over $K$ as $\mathcal{D}$ (e.g.\ if $P\in{X(K)}$ is a periodic point of exact period $n$ of $f$, then $\sigma^{-1}(P)$ is a periodic point of exact period $n$ of $f^\sigma$). Moreover, $\mathcal{D}$ is strongly residually periodic if and only if $\mathcal{D}^\sigma$ is strongly residually periodic, as their dynamical behavior may differ at only finitely many primes.

It is not hard to check that all the automorphisms of the Markoff surface of type $\phi_i\phi_j$, for distinct $ i,j\in\{1,2,3\}$, are linearly conjugate by permuting the variables $x,y$ and $z$. Therefore it is enough to check strong residual periodicity only for $\phi_1\phi_2$. 

\section{The global picture} \label{sec:global}

As mentioned in the introduction, it is already known by Cantat and Loray~\cite[\S{5.1}]{article:cantat-loray2009} that the automorphism $\phi_1\phi_2$ has no periodic points defined over $\mathbb{R}$, and therefore over $\mathbb{Q}$, other than $(0,0,0)$. More specifically, they proved that a point $(x,y,z)$, defined over $\mathbb{C}$, is periodic if and only it is fixed (i.e.\ the point $(0,0,0)$) or if $z=\pm{2cos(\pi\theta)}\neq{\pm{2}}$, where $\theta$ is a rational number (compare with table \ref{tab:periodic_fibers} below). However, we carefully study the dynamics over $\mathbb{Q}$, as it will help us in the local analysis as well.

The automorphism $\phi_1\phi_2$ fixes the third parameter $z$, and the projection $\pi_3 : (x,y,z)\mapsto z$ induces a conic fibration on $M$, which is \emph{preserved} by $\phi_1\phi_2$; i.e.\ each fiber is invariant under $\phi_1\phi_2$. The action of the automorphism on each fiber is linear. 

Fixing a fiber $z=\lambda$, we get 
\begin{equation}
\phi_1\phi_2(x,y,\lambda) =
\begin{bmatrix}
-1 & \lambda \\
0 & 1
\end{bmatrix}
\begin{bmatrix}
1 & 0 \\
\lambda & -1
\end{bmatrix} 
\begin{bmatrix}
x\\y
\end{bmatrix}
=
\begin{bmatrix}
-1+\lambda^2 & -\lambda \\
\lambda & -1
\end{bmatrix}
\begin{bmatrix}
x\\y
\end{bmatrix}.
\end{equation}

We denote
\begin{equation}
A_\lambda = \begin{bmatrix}
-1+\lambda^2 & -\lambda \\
\lambda & -1
\end{bmatrix}.
\end{equation}

A point $(x,y)\neq{(0,0)}$ on the fiber over $z=\lambda$ is a periodic point of $A_\lambda$ if and only if $\det(A_\lambda^n-I) = 0$ for some positive integer $n$. We notice that the determinant of $A_\lambda$ is $1$, so that $\det(A_\lambda^n)=1$. Suppose that $UAU^{-1}$ is the Jordan form of $A$, then $\det(A_\lambda^n-I) = 0$ implies $\det((UA_{\lambda}U^{-1})^n-I)=0$. This in turn implies that one of the eigenvalues of $A_\lambda$ is a root of unity or zero, and since the determinant of $A_\lambda$ is $1$ we get that both eigenvalues must be roots of unity. Thus the possible Jordan forms for $A_\lambda$ are:
\begin{equation} \label{eq:jordan}
\begin{bmatrix} 
\mu & 1 \\
0 & \mu
\end{bmatrix}, 
\quad 
\begin{bmatrix}
\mu & 0\\
0 & \mu^{-1}
\end{bmatrix}
\end{equation}
where $\mu$ is a root of unity and $\mu^{-1} = \overline{\mu}$, the complex conjugate of $\mu$. Such an $A_\lambda$ has infinite order if and only if its Jordan form is of the first type in \eqref{eq:jordan}, and since $\det(A_\lambda)=1$ this implies $\mu=\pm{1}$.

The characteristic polynomial of $A_\lambda$ for any $\lambda$ is
\begin{equation}
f_\lambda(t) = \det(tI-A_\lambda) = t^2+(2-\lambda^2)t+1 = t^2+dt+1
\end{equation}
where $d=2-\lambda^2=-Tr(A_\lambda)$.
If $\lambda$ is rational then so is $d$, so that the minimal polynomial of a root $\mu$ of $f_\lambda(t)$ must divide $f_\lambda(t)$. Since $\mu$ is a root of unity, the minimal polynomial is a cyclotomic polynomial of degree at most $2$, and there are only finitely many of those. We summarize them in the following table (here $\rho_n$ is a chosen primitive $n$-th root of unity in $\overline{\mathbb{Q}}$):
\begin{table}[h]
\begin{center}
\caption {Fibers where the roots of $f_\lambda$ are roots of unity and $d$ is rational} \label{tab:periodic_fibers} 
\begin{tabular}{| >{$}l<{$} | >{$}l<{$} | >{$}l<{$} | >{$}l<{$} | >{$}l<{$} |}
\hline
\mu & f_\lambda & d & \lambda  \\ \hline
1 & t^2-2t+1 & -2 & \pm{2} \\ 
-1 & t^2+2t+1 & 2 & 0\\ 
\rho_3 & t^2+t+1 & 1 & \pm{1}\\ 
\rho_4 & t^2+1 & 0 & \pm\sqrt{2}\\
\rho_6 & t^2-t+1 & -1 & \pm\sqrt{3} \\ \hline
\end{tabular} 
\end{center}
\end{table}

We see that for the cases of $\rho_4$ and $\rho_6$ the fibers over $\lambda$ are not defined over $\mathbb{Q}$, so they cannot contribute rational periodic points. 

If $\mu=1$ then $\lambda=\pm{2}$, and
$A_\lambda = 
\begin{bmatrix}
3 & \mp{2} \\
\pm{2} & -1
\end{bmatrix}$ with Jordan form 
$\begin{bmatrix}
1 & 1 \\
0 & 1
\end{bmatrix}$. Such an $A_\lambda$ can only have periodic points that are fixed, and these can be checked to have to be of the form $(x,x)$ for $\lambda=2$, and $(x,-x)$ for $\lambda=-2$. In any case, the fibers over $\lambda=\pm{2}$
are
$(x\mp{y})^2+{4} = 0,$
and these are reducible over $\mathbb{Q}(i)$ and have no $\mathbb{Q}$-rational points. 

For $\mu=-1$ we get $A_\lambda = -I$, but the equation of the conic $x^2+y^2 = 0$ is reducible over $\mathbb{Q}(i)$ and the only $\mathbb{Q}$-rational point is $(0,0,0)$, which is fixed.

Finally, for the case $\mu = \rho_3$, the two fibers over $\lambda = \pm{1}$ are 
\begin{equation}
x^2+y^2\mp{xy}+1 = 0.
\end{equation}
These curves are geometrically irreducible and contain no rational points.

By the computations above we can conclude that the system $\mathcal{D}_1$ (see Theorem~\ref{thm:main-thmA}) is SRP(3): First, the fibers over $\lambda = \pm{1}$ contain rational points over $\mathbb{F}_p$ for any prime $p\neq{3}$. In fact, by using for instance Chevalley's theorem (cf.\ Ireland and Rosen~\cite[\S{10.2}]{book:ireland-rosen1990}), we know that for any prime $p$ there exists a $\mathbb{F}_p$-rational point in the projective closure of the plane curves $x^2+y^2\pm{}xy+1$ (the point might be at infinity); since the curves are smooth for $p\neq{3}$ they are rational, and we get that the curves should contain $p+1$ rational points, at least one of which is not at infinity (for the case of $p=3$ one can check that there are no solutions not at infinity). Second, the automorphism $\phi_1\phi_2$ is of order $3$ on these fibers as $A_\lambda=\begin{bmatrix}
\rho_3 & 0\\
0 & \rho_3^{-1}
\end{bmatrix}$ is of order $3$; therefore the $\mathbb{F}_p$-rational points on the fibers will be periodic of period dividing $3$; note that since the only fixed point on the surface is $(0,0,0)$, the $\mathbb{F}_p$-rational points on the fibers will be periodic of exact period $3$. This proves the first part of Theorem \ref{thm:main-thmA}. Note that we have not yet ruled out the existence of periodic points of period 1 and 2 modulo all but finitely many primes, this will be done in the next section.

\section{The local picture} \label{sec:local}

Taking the determinant commutes with reducing modulo $p$ (since the determinant is a polynomial in the entries of the matrix), so that $A_\lambda$ has determinant $1$ modulo every prime $p$. The same reasoning as in the previous section implies that we are interested in pairs $(t,d)\in\overline{\mathbb{F}_p}^*\times\mathbb{F}_p$, where $\overline{\mathbb{F}_p}^*$ is the group of invertible (nonzero) elements in the algebraic closure of $\mathbb{F}_p$, satisfying 
\begin{equation} \label{eq:local-equation}
 t^2+dt+1 = 0.
\end{equation}
Let $(t,d)\in\overline{\mathbb{F}_p}^*\times\mathbb{F}_p$ be a solution to \eqref{eq:local-equation}, and suppose that $t$ has order $n$ in $\overline{\mathbb{F}_p}^*$. Then $t$ is a primitive $n$-th root of unity in $\overline{\mathbb{F}_p}$, so it is also a root of the $n$-th cyclotomic polynomial $\Phi_n$ (the $n$-th cyclotomic polynomial $\Phi_n$ is the minimal polynomial of the primitive $n$-th roots of unity over $\mathbb{Q}$). We can assume that $\gcd(p,n)=1$, since otherwise there are no primitive $n$-th roots in $\overline{\mathbb{F}_p}$.

We recall the following fact from the algebra of finite fields (see Lidl and Niederreiter~\cite[Theorem 2.47]{book:lidl-niederreiter1986}): suppose $\gcd(p,n)=1$, and let $k$ be the order of $p$ in the multiplicative group of invertible elements in the ring $\mathbb{Z}/n\mathbb{Z}$; then $\Phi_n$ factors over $\mathbb{F}_p$ into $\frac{\phi(n)}{k}$ irreducible polynomials of degree $k$ (where $\phi$ is Euler's totient). 

We are interested in the minimal $n$ such that $t^2+dt+1$ and $\Phi_n$ have a common factor in the ring $\mathbb{F}_p[t]$; suppose that this is so: if $t^2+dt+1$ has distinct roots then $t^2+dt+1$ is a factor of $\Phi_n$, otherwise it has a multiple roots that must be $\pm{1}$. In the latter case we get either $f_\lambda(t) = (t-1)^2=\Phi_1(t)^2$ or $f_\lambda(t) = (t+1)^2 = \Phi_2(t)^2$. By the fact mentioned in the previous paragraph, $k$ (the order of $p$ modulo $n$) must be either $1$ or $2$, so that $p^2\equiv{1}\pmod{n}$, or in other words, $n$ is a divisor of $p^2-1$. We have thus proved the following lemma:

\begin{lem} \label{lem:one}
If the residual system $\mathcal{D}_{2,p}$ (see Theorem \ref{thm:main-thmB}) has a periodic point of exact period $n$, then $n$ divides $p^2-1$.
\end{lem}

It is easy to check that for any prime $p\geq{5}$ we have $24|(p^2-1)$. We denote by $\omega(p)$ the smallest prime factor of $\frac{p^2-1}{24}$ for a prime $p\geq{5}$.

\begin{lem} \label{lem:two}
If $p\geq{5}$ is a prime of the form $p=24\ell+19$, then any periodic point of the residual system $\mathcal{D}_{2,p}$ lying outside of the forbidden set $\mathcal{F}_{2,p}$ has period greater or equal to $\omega(p)$.
\end{lem}

\begin{proof}
By Lemma \ref{lem:one} it is enough to prove that $\mathcal{D}_{2,p}$ has no periodic points of period dividing $24$ for primes $p$ of the form $p=24\ell+19$, lying outside the forbidden set. We proceed to analyze all cases for positive integers $n$ that are factors of $24$, and check for which fibers $\lambda$ the polynomial $f_\lambda(t)$ has a common factor with the cyclotomic polynomial $\Phi_n(t)$.


For $n=1$, we must have $f_\lambda(t) = \Phi_1(t)^2 = t^2-2t+1$. Therefore $d=-2$ and $\lambda=\pm{2}$ (see \S\ref{sec:global}). The fibers over $\lambda=\pm{2}$ have $\mathbb{F}_p$-rational points only if $-1$ is a quadratic residue modulo $p$. The Jordan form of $A_\lambda$ (for $p>2$) is
\begin{equation}
\begin{bmatrix}
1 & 1 \\
0 & 1
\end{bmatrix}
\end{equation}
which is of infinite order over $\mathbb{Q}$, but of order $p$ over $\mathbb{F}_p$. One can see from the Jordan form that if there are periodic points of period $m<p$, they must be fixed of the form $(x,0)$. For $\lambda=\pm{2}$
\begin{equation}
A_\lambda = \begin{bmatrix}
3 & \mp{2} \\
\pm{2} & -1
\end{bmatrix},
\end{equation}
and one can check that there are no $\mathbb{F}_p$-fixed points over the relevant fibers.

For $n=2$, we must have $f_\lambda(t) = \Phi_2(t)^2 = t^2+2t+1$, implying that $\lambda=0$; the fiber over $\lambda=0$ is geometrically reducible, and other than the point $(0,0,0)$ contains $\mathbb{F}_p$-rational points only when $-1$ is a quadratic residue modulo $p$. This means that for primes of the form $q=4\ell+1$ there exist points of exact period $2$ for large enough primes $q$. However, primes of the form $p=24\ell+19$ are not of this form, so for these primes there are no periodic points of period $2$. 

For $n=3$, we have already seen in the previous section that the fibers over $\lambda=\pm{1}$ will contain points of exact period $3$ for large enough primes $p$. These fibers are in the (reduction of the) forbidden set of $\mathcal{D}_2$.

For $n=4$ and $n=6$, we have seen in the previous section that when $2$ and $3$ respectively are quadratic residues (for primes of the form $p=8\ell\pm{1}$ and $p=12\ell\pm{1}$, respectively), the relevant fibers will be defined over $\mathbb{F}_p$; it is easy to check that for large enough primes they are geometrically irreducible, and will contain $\mathbb{F}_p$-periodic points of exact period $4$ and $6$. However, for primes of the form $p=24\ell+19$, the numbers $2$ and $3$ are not quadratic residues.

For $n=8$, the only type of factor of $\Phi_8(t) = t^4+1$ that is of the form $t^2+dt+1$ is $t^2+\alpha{t}+1$, where $\alpha^2 = 2$. Since $\alpha$ must be in $\mathbb{F}_p$, we get this type of factor only for primes $p$ such that $2$ is a quadratic residue, that is primes of the form $p=8\ell\pm{1}$. We remark here that the existence of a factor of the form $t^2+dt+1$ does not guarantee $\mathbb{F}_p$-periodic points of period $8$, since the fibers may still not be defined over $\mathbb{F}_p$, so that the minimal periods may be larger.

For $n=12$, the cyclotomic polynomial $\Phi_{12}(t) = t^4-t^2+1$ has a factor of the form $t^2+dt+1$ for $d=\beta$ where $\beta^2 = 3$, but $3$ is a quadratic residue mod $p$ if and only if $p= 12\ell\pm{1}$.

For $n=24$, the polynomial $\Phi_{24}(t) = t^8-t^4+1$ has a factor of the required form if and only if $d=\gamma$ where $\gamma$ is a root of $x^4-4x^2+1$. But this polynomial has a root only for those primes $p$ where $2$ and $3$ are \emph{both} quadratic residues mod $p$ (the splitting field of $x^4-4x^2+1$ over $\mathbb{Q}$ is $\mathbb{Q}(\sqrt{2},\sqrt{3})$), which does not happen for primes $p=24\ell+19$.
\end{proof}

\begin{lem} \label{lem:three}
There exists an increasing sequence of primes $(p_k)_{k=1}^{\infty}$ such that $p_k$ is of the form $p_k=24\ell+19$ for any positive integer $k$, and $$\lim_{\substack{k\rightarrow\infty}}\omega(p_k)=\infty.$$
\end{lem}

As remarked to the author by Roger Heath-Brown, the proof of this lemma is an easy combination of the Chinese remainder theorem and Dirichlet's theorem.

\begin{proof}
To prove the lemma it is enough to show that for any positive integer $k$ there exists a prime $p_k$ such that $\omega(p_k)>k$. Let $q_1=2,q_2=3,...,q_r$ be the list of primes up to $k$, and let $q=\prod_{j=3}^r q_j$ be the product of all but the first two. The numbers $q$ and $6$ are coprime, so by the Chinese remainder theorem (see Ireland and Rosen~\cite[\S{3.4}]{book:ireland-rosen1990}) there exists a solution $x=a$ to the system of congruences
\begin{equation*}
\begin{cases}
x \equiv 5 \pmod{72} \\
x \equiv 2 \pmod{q}
\end{cases}.
\end{equation*}

Since $a$ and $72q$ must be coprime, there exists by Dirichlet's theorem (see Serre \cite[\S{VI.4}]{book:serre1973}) a prime $p_k$ such that $p_k\equiv{a}\pmod{72q}$. Now $4 | p_k-1$ but $\frac{p_k-1}{4} \equiv 1 \pmod{18}$ and $6 | p_k+1$ but $\frac{p_k+1}{6} \equiv 1 \pmod{12}$. Hence $\frac{p_k^2-1}{24}$ is coprime to $6$. Moreover, for any prime $5\le{p}\le{k}$ we have $p_k\equiv{a}\equiv{2} \pmod{p}$ so that $p_k-1\equiv{1}\pmod{p}$ and $p_k+1\equiv{3}\pmod{p}$; thus $\frac{p_k^2-1}{24}$ is coprime to $p$. This implies that $\omega(p_k)>k$.
\end{proof}

Theorem \ref{thm:main-thmB} now follows from Lemmas \ref{lem:two} and \ref{lem:three}: Let $(p_k)_{k=1}^{\infty}$ be a sequence as in Lemma \ref{lem:three}; by Lemma \ref{lem:two} the minimal periods of the residual systems $\mathcal{D}_{2,p_k}$ tend to infinity with $k$, so that the system $\mathcal{D}_2$ cannot be strongly residually periodic.

We can also conclude that the dynamical system $\mathcal{D}_1$ is not SRP(1) or SRP(2), since in the proof of Lemma \ref{lem:two} we can see there exist periodic points for the residual system $\mathcal{D}_{1,p}$ of exact period $2$ only for large enough primes of the form $p=4\ell+1$, and no fixed points other than $(0,0,0)$. Together with the conclusion of the previous section, we have proved Theorem~\ref{thm:main-thmA}.

\subsection*{Acknowledgments}
I would like to thank Tatiana Bandman, Fabrizio Barroero, Jung Kyu Canci, Serge Cantat, Roger Heath-Brown, Lars K\"uhne, Boris Kunyavski{\u\i} and Umberto Zannier for their comments and corrections, St\'ephane Lamy for example \eqref{eq:lamy} and the referee for the careful reading, many corrections and helpful comments.

The research for the article was supported by the ERC-Advanced Grant ``Diophantine Problems" (Grant no.\ 267273).


\def\cprime{$'$}

\end{document}